\documentclass{amsart}

\usepackage{amssymb}
\usepackage{amsmath}
\usepackage{amsthm}
\usepackage[mathscr]{eucal}

\usepackage[active]{srcltx}

\newtheorem{thrm}{Theorem}[section]
\newtheorem{lemma}[thrm]{Lemma}
\newtheorem{prop}[thrm]{Proposition}
\newtheorem{cor}[thrm]{Corollary}

\theoremstyle{definition}

\theoremstyle{remark}

\numberwithin{equation}{section}

\newcommand{\dbar}{$\bar{\partial}$}
\newcommand{\dq}{\overline{\partial}}
\newcommand{\mdbar}{\bar{\partial}}
\newcommand{\R} {\ensuremath{\mathbb{R}}}

\newcommand{\C} {\ensuremath{\mathbb{C}}}

\begin{document}

\title[Subelliptic estimates on a singular spaces]
 {Subelliptic estimates
for the \dbar-problem on a singular complex space}

\author[D. Ehsani, J. Ruppenthal]
 {Dariush Ehsani, Jean Ruppenthal}

\address{
Bergische Universit\"{a}t Wuppertal\\
Fachbereich C - Mathematik \& Naturwissenschaften\\
D-42119 Wuppertal}
 \email{dehsani.math@gmail.com}
 \email{jean.ruppenthal@math.uni-wuppertal.de}

\begin{abstract}
We prove subelliptic estimates for the $\dq$-problem at the isolated singularity
of the variety $z^2=xy$ in $\C^3$.

\end{abstract}

\maketitle

\bibliographystyle{plain}

\section{Introduction}

The Cauchy-Riemann operator $\dq$ and the related \dbar-Neumann operator play a central role in complex analysis.
Especially the $L^2$-theory for these operators is of particular importance and has become indispensable for the subject
after the fundamental work of
H\"ormander on $L^2$-estimates and existence theorems for the \dbar-operator \cite{Hoe1}
and the related work of Andreotti and Vesentini \cite{AnVe}.
By no means less important is Kohn's solution of the $\dq$-Neumann problem,
which implies existence and regularity results for the $\dq$-complex, as well (see \cite{FK}).
Important applications of the $L^2$-theory are e.g. the Ohsawa-Takegoshi extension theorem \cite{OT},
Siu's analyticity of the level sets of Lelong numbers \cite{Siu0}
or the invariance of plurigenera \cite{Siu} -- just to name some.

Whereas the theory of the $\dq$-operator and the $\dq$-Neumann operator is very well developed on
complex manifolds, still not too much is known about the situation on singular complex spaces which
appear naturally as the zero sets of holomorphic functions.
The further development of this theory is an important endeavor
since analytic methods have led to fundamental advances in geometry on complex manifolds,
but these analytic tools are still missing on singular spaces.

The topic has attracted some attention recently and considerable progress has been made.
Let us mention e.g. the development of some Koppelman formulas by Andersson and Samuelsson (\cite{AS1}, \cite{AS2}).
Concerning the $L^2$-theory for the $\dq$-operator, {\O}vrelid and Vassiliadou
obtained essential results for the case of isolated singularities.
Following a path prepared by 
Forn{\ae}ss, Diederich, Vassiliadou and {\O}vrelid
(\cite{Fo}, \cite{DFV}, \cite{FOV2}, \cite{OV1}, \cite{OV2}) and by
Ruppenthal and Zeron (see \cite{Rp1}, \cite{Rp7}, \cite{Rp8}, \cite{RZ1}, \cite{RZ2}),
they were finally able to describe the $L^2$-cohomology for the $\dq$-operator
at isolated singularities completely in terms of a resolution of singularities (see \cite{OV3}).
For another, different approach to these results we refer also to \cite{Rp10}.

These works mark important progress concerning the understanding
of the obstructions to solving the $\dq$-equation at isolated
singularities. It remains to study the regularity of the equation:
on domains in complex manifolds, the close connection between the
regularity of the $\dq$-equation on one hand and the geometry of
the domain (and its boundary) on the other hand is one of the
central topics of complex analysis. Recall that a domain in $\C^n$
with smooth pseudoconvex boundary is of finite type exactly if the
$\dq$-Neumann problem is subelliptic, and that there is a deep
connection between the type of the boundary and the order of
subellipticity (cf. the works of Kohn, Catlin and D'Angelo). It is
an interesting endeavor to establish such connections also between
the regularity of the $\dq$-equation at singularities and the
geometry of the singularities.

Compactness of the $\dq$-Neumann operator, which can be seen as a boundary case of subelliptic regularity,
has been established at isolated singularities recently in \cite{Rp9} and \cite{OR}.
Compactness of the $\dq$-Neumann operator is an important property
in the study of weakly pseudoconvex domains. 
Moreover, it yields that the corresponding space of $L^2$-forms
has an orthonormal basis consisting of eigenforms of the $\dq$-Laplacian $\Box=\dq\dq^* +\dq^*\dq$.
The eigenvalues of $\Box$ are non-negative, have no finite limit point
and appear with finite multiplicity.

In the present paper, we make the next step and study subelliptic estimates for the $\dq$-problem
at an isolated singularity.
Besides the general question whether is is possible to classify singularities by the degree of subelliptic
estimates which hold, the main motivation is as follows.
It is one of the key observations in \cite{FOV2}, \cite{Rp7}, \cite{OV3}, \cite{Rp10}
that the $\dq$-equation can be solved in the $L^2$-category at isolated singularities with some gain of regularity,
and one reason why the theory is not well developed for arbitrary singularities is that such results do not yet
exist in that case (one can only solve in the $L^2$-category but something better is needed).
In view of such questions it is natural to consider the canonical $\dq$-solution operator $\dq^* N$,
where $N$ is the $\dq$-Neumann operator, and to study regularity of $N$ and $\dq^*N$.

\medskip
To begin with, it makes sense to consider a simple example.
So, we decided to study the isolated singularity of the variety $Z:=\{z_3^2=z_1 z_2\}$ in $\C^3$.
Let $X$ be the intersection of $Z$ with the cylinder
  $\{|z_1|^2+|z_2|^2<1\}$ in $\C^3$.  We note that
  $X$ has a strictly pseudoconvex boundary
and consider $X$ as a Hermitian complex space with the restriction of the Euclidean metric of $\C^3$.
We define the weighted $L^2$-spaces of $(p,q)$-forms on $X$:
\begin{equation*}
 L^{2,k}_{(p,q)}(X)=\left\{ f: \|f\|_{L^{2,k}}^2=\int_X\gamma^{2k}
  |f|^2 dV_X <\infty \right\},
\end{equation*}
with the weight
\begin{equation*}
\gamma = \sqrt{|z_1|+ |z_2|}.
\end{equation*}
Note that $\gamma^2 \sim \|z\|= \sqrt{|z_1|^2+|z_2|^2+|z_3|^2}$ on
$X$. For ease of notation, we will write simply $\|f\|$ for
$\|f\|_{L^{2,0}}$.

We also make use of some Sobolev spaces on $X$.  To introduce
these spaces, we consider the $2$-sheeted covering $\pi:
\C^2\rightarrow Z, (v,w) \mapsto (z_1, z_2, z_3)=(v^2,w^2,vw)$ for
which $\pi^* dV_X \sim (\pi^* \gamma^4) dV_{\C^2}$. We say that a
function $f$ is in $W^k(X)$ if all the partial derivatives up to
order $k$ of $\pi^* f$ with respect to the $(v,w)$-coordinates are
square-integrable with respect to the volume $(\pi^* \gamma^4)
dV_{\C^2}$ on $\C^2$ (see \eqref{sobolev1}). Non-integer Sobolev
spaces $W^\epsilon(X)$ are then defined as usually by use of the
Fourier transform on $\C^2$ (see \eqref{sobolev2}). For
differential forms, we define the $W^\epsilon$-norm as the sum of
the $W^\epsilon$-norms of its coefficients in an orthonormal
coordinate system.

Our first main result is as follows:

\begin{thrm}
\label{weightestI}
 Let $f\in
L^{2,-1}_{(0,1)}(X)\cap\mbox{dom}(\mdbar)\cap
\mbox{dom}(\mdbar^{\ast})$ with support in a neighborhood of the
singularity. Then we have the estimate
\begin{equation*}
\|f\|_{W^1(X)}^2
 \lesssim \| \mdbar f\|^2 + \| \mdbar^{\ast} f\|^2
  +\|f\|^2_{L^{2,-1}}.
\end{equation*}
\end{thrm}

Theorem \ref{weightestI} is a special case of our more general estimate \eqref{LL}.
 Note that particularly bounded forms are in $L^{2,-2}(X)$ as $\|z\|^{-2}$
is integrable on $X$, i.e. $\|f\|_{L^{2,-2}} \lesssim
\|f\|_{L^\infty}$. Theorem \ref{weightestI} shows that we can have
full regularity for forms with a certain extra-vanishing at the
singularity. This makes sense as it is known that there exist
usually finitely many obstructions to solving the $\dq$-equation
in the $L^2$-category at isolated singularities and that the
number of obstructions to solving $\dq u=f$ is decreasing if $f$
is in $L^p$ for increasing $p$ (see \cite{Rp7}, Theorem 1.1 and
Theorem 1.2).

In this spirit, we derive a type of subelliptic estimates for the
$\dq$-problem which can be viewed as a trade-off of a decrease in
derivatives for less stringent
 hypotheses on to which $L^p$ spaces a given form may belong.
 Our results show that the restriction of
  a form $f$ from more $L^p$ spaces for $p>2$
  allows for higher orders of derivatives of $f$
 to be estimated.

Then our second main result reads as:

\begin{thrm}
 Let $0\le \epsilon\le 1/2$, $p>4/(2-\epsilon)$,
  and $f\in L^{p}_{(0,1)}(X)\cap
 \mbox{dom}(\mdbar)\cap \mbox{dom}(\mdbar^{\ast})$.  Then the
 following
 subelliptic estimate holds
\begin{equation*}
 \| f\|_{W^\epsilon(X)}
  \lesssim
\|  \mdbar f\|+ \| \mdbar^{\ast} f\|
 +\| f\|_{L^{p}(X)}
  .
\end{equation*}
\end{thrm}

Here, the restriction to $\epsilon\leq 1/2$ is only due to the
fact that $X$ has a strongly pseudoconvex boundary. For forms with
compact support in $X$ the statement is valid for all $0\leq
\epsilon\leq 1$, by Theorem \ref{weightestI}.

\section{Preliminary calculations}
\label{prelim}
  We note that the variety $X$ can be parametrized by the two sheeted covering
\begin{equation*}
\pi: \C^2\rightarrow X,\ (v,w) \mapsto (v^2,w^2,vw).
\end{equation*}
Instead of working directly on $X$, we consider instead the
strongly pseudoconvex domain $B=\{(v,w):|v|^4+|w|^4<1\}$ in $\C^2$
with the metric
\begin{equation}
\label{metric}
 (g_{ij})=\left(
       \begin{array}{cc}
        4|v|^2+|w|^2 &  v\overline{w}  \\
          \overline{v}w & |v|^2+4|w|^2
       \end{array}
     \right),
\end{equation}
so that
the volume element is given by
\begin{equation}
 \label{volume}
dV_g = \det (g_{ij}) dv\wedge dw \wedge d\bar{v}\wedge d\bar{w}.
\end{equation}
Then we have simply
\begin{equation*}
\|f\|^2_{L^{2,k}} = \frac{1}{2} \int_B (\pi^* \gamma)^{2k} |\pi^*
f|^2 dV_g
\end{equation*}
for the weighted $L^2$-norms on $X$. So, for the questions that we
study in the present paper, it is absolutely sufficient to replace
the original variety $X$ in $\C^3$ by the smooth domain $B$ in
$\C^2$ with the positive semi-definite pseudometric $g=(g_{ij})$.

Thus, from now on, let $X$ be the Hermitian space $(B,g)$. It is
our goal to study subelliptic estimates for the $\mdbar$-problem
on
the Hermitian complex space $X$.\\

 We can calculate $\mdbar$ and $\mdbar^{\ast}$
 in the holomorphic coordinates $v$ and $w$ (see \cite{Mue}).  For
 $f\in L^2(X)$ we have
\begin{equation*}
 \mdbar f
  =\frac{\partial f}{\partial \bar{v}} d\bar{v}
   +\frac{\partial f}{\partial \bar{w}} d\bar{w},
 \end{equation*}
and in the case $f=f_1d\bar{v}+f_2 d\bar{w}$, $f\in
L^2_{(0,1)}(X)$, we have
\begin{equation*}
 \mdbar f =
  \left(\frac{\partial f_2}{\partial \bar{v}}
   -\frac{\partial f_1}{\partial \bar{w}}\right) d\bar{v}\wedge d\bar{w}.
\end{equation*}

To describe the operator $\mdbar^{\ast}$ we first define
\begin{equation*}
|g|=\det(g_{ij})
 = 16|v|^2|w|^2 +4|v|^4 +4 |w|^4.
\end{equation*}
 Thus we can write
\begin{equation*}
(g^{ij})=\frac{1}{|g|}
 \left(
       \begin{array}{cc}
        |v|^2+4|w|^2 &  -v\overline{w}  \\
          -\overline{v}w & 4|v|^2+|w|^2
       \end{array}
     \right).
\end{equation*}

For the operator $\mdbar^{\ast}$ acting on a $(0,1)$ form,
$f=f_1d\bar{v}+f_2d\bar{w}$ we can write
\begin{equation*}
\label{star1}
 \mdbar^{\ast} f
 = \frac{1}{|g|}
  \left( \frac{\partial}{\partial v} |g| g^{11}f_1
   +\frac{\partial}{\partial w} |g| g^{12}f_1 +
   \frac{\partial}{\partial v} |g| g^{21}f_2
   + \frac{\partial}{\partial w} |g| g^{22}f_2\right).
\end{equation*}

  From the discussion above, we immediately see that
 singularities occur in the operators under study at
  $v=w=0$.  We make our
  calculations with respect to an orthonormal
  frame of $(1,0)$ and $(0,1)$ forms and
  we keep track of the singularities arising at
  $v=w=0$ in the use of such forms.
    As we shall see, we
  need to consider fields with coefficients with singular
  behavior near the singularity of $X$.  We keep
  track of the blow up near the singularity in terms of the factor
 \begin{equation*}
 \gamma = \sqrt{|v|^2+|w|^2}.
 \end{equation*}

Note that $\gamma^2$ behaves like the distance to the origin in $X$ (with respect to the metric $g$),
because $\gamma^2=|z_1|+|z_2|$ when we consider $\gamma$ in the $z$-coordinates on the original variety in $\C^3$
carrying the restriction of the Euclidean metric.\\

  A
 simple calculation shows an
orthonormal system of $(1,0)$ forms will consist of forms written
as
\begin{equation*}
\alpha_1(v,w) dv + \alpha_2(v,w) dw,
\end{equation*}
where
 \begin{equation*}
 |\alpha_i| \sim \gamma^{} \qquad i=1,2
\end{equation*}
 and hence the dual frame consists of vectors which can be
written as
\begin{equation}
 \label{vfform}
\beta_1(v,w)\frac{\partial}{\partial v}+
\beta_2(v,w)\frac{\partial}{\partial w},
\end{equation}
where
\begin{equation*}
|\beta_i| \sim \frac{1}{\gamma^{}} \qquad i=1,2.
\end{equation*}
We follow the notation in \cite{EhLi10} to keep track of the
singularity near the origin.
  We write $\xi_k$
 to denote an operator which on the level of functions
 is the multiplication
   by a function $\xi_k(\zeta)$
  with the
 property
\begin{equation}
 \label{xiprop}
 |\gamma^{\alpha}D^{\alpha}_{\zeta}
 \xi_k(\zeta)|\lesssim \gamma^k,
\end{equation}
where $D^{\alpha}_{\zeta}$ is a differential operator, using the
index notation, of order $|\alpha|$: for a multi-index
$\alpha=(\alpha_1,\alpha_2)$, where each $\alpha_j$ is an integer,
we use the notation
\begin{align*}
&D^{\alpha}=\frac{\partial^{\alpha_1} }{\partial
\zeta_1^{\alpha_1}}
 \frac{\partial^{\alpha_2} }{\partial
\zeta_2^{\alpha_2}}\\
&|\alpha|=\alpha_1+\alpha_2 .
\end{align*}
On the level of forms, $\xi_k f$ is a sum of terms
  which are the product of coefficients of $f$
  with forms whose coefficients are of type $\xi_k$ on the level
  of functions.
Thus, for instance with $f=f_1\omega_1 + f_2\omega_2$,
 $\xi_k f$ could be used to denote a $(0,1)$-form such as
  $ \xi_k^1 f_1 \omega_1 + \xi_k^2f \omega_2$, where
 $\xi_k^1$ and $\xi_k^2$ satisfy estimates as in (\ref{xiprop}),
 or $\xi_k f$ could be used to denote a function
  $\xi_k^1 f_1 + \xi_k^2f_2$.

We choose an orthonormal frame $\omega_1$, $\omega_2$ of $(1,0)$
vectors as above, and denote by $L_1$ and $L_2$ the dual frame.
 In the next section we will see the following two calculations
 come into play.  From the discussion above we
  compute
 \begin{align}
 \label{domeg}
&\mdbar \omega^i = \sum_{j,k} \xi_{-2}
 \overline{\omega}^j\wedge\omega^k\\
  \label{comm}
& [ L_j,\overline{L}_k ]
  = \sum_i \xi_{-2} L_i +
   \sum_i \xi_{-2} \overline{L}_i.
\end{align}

 \section{Integration by parts}
  We take as our guide the situation on a smoothly bounded strictly
  pseudoconvex domain, $\Omega\subset\subset\mathbb{C}^2$.
  In such a setting, for $f=f_1 d\bar{z}_1 + f_2d\bar{z}_2$, where
  $f_i \in C^1(\overline{\Omega})$, we have the
  Morrey-Kohn-H\"{o}rmander formula (see \cite{CS}):
\begin{equation}
\label{mkh}
 \|\mdbar f\|^2 + \|\mdbar^{\ast} f \|^2
 = \sum_{i,j=1}^2\int_{\Omega}
  \left| \frac{\partial f_i}{\partial \bar{z}_j} \right|^2 dV
 +\sum_{i,j=1}^2\int_{\partial\Omega}
  \frac{\partial^2 \rho}{\partial z_i \partial \bar{z}_j}
   f_j \overline{f}_j dS,
\end{equation}
where $dV$ refers to the volume element,  $dS$ the surface area
element, and $\rho$ is a defining function for the domain
$\Omega$.  An important fact to be used later is that strict
pseudoconvexity implies the Levi form associated with the boundary
is strictly positive definite and thus allows the last
 term on the right to be bounded from below.
 The relation in (\ref{mkh}) is obtained by an integration by parts in
the norms $\|\mdbar f\|^2$ and $\|\mdbar^{\ast} f\|^2$. We shall
see below that an integration by parts leads to weighted norms,
the weights of which blow up at our singularity
  due to
the singular nature of the operators.
We consider the norms $ \|\mdbar f\|^2$ and $\|\mdbar^{\ast} f
\|^2$ in terms of the vector fields $L_1$, $L_2$,
$\overline{L}_1$, and $\overline{L}_2$.

We write a $(0,1)$-form, $f$, as $f=f_1\overline{\omega}_1 + f_2
 \overline{\omega}_2$.
  For $f$ a  $(0,1)$-form which is $C^1(\mathbb{C}^2)$
   (in the sense the derivatives with respect to coordinates
    given by $v$ and $w$ in Section \ref{prelim} are
     continuous) with compact
support in a neighborhood of the singularity, we can write, using
 (\ref{domeg}),
\begin{equation}
 \label{bar}
\mdbar f = \left(\overline{L}_1 f_2 -
 \overline{L}_2 f_1\right) \overline{\omega}_1\wedge
 \overline{\omega}_2 + \xi_{-2} f,
\end{equation}
and similarly, using integration by parts and the
 fact that the vector fields, $L_1$ and $L_2$ have singular
 coefficients as in (\ref{vfform}), we have
\begin{equation}
\label{star}
 \mdbar^{\ast} f
 = L_1 f_1
  +
L_2 f_2
  + \xi_{-2} f .
\end{equation}

We define the norm
\begin{equation*}
 \|\overline{L} f\|^2
  =\sum_{i,j=1}^2
 \|\overline{L}_j f_i\|^2.
\end{equation*}
Likewise, the norm
\begin{equation*}
 \|L f\|^2
  =\sum_{i,j=1}^2
 \|L_j f_i\|^2
\end{equation*}
will be used.  Through an integration by parts and the commutator
relation (\ref{comm}) above the two norms are related:
\begin{align}
\nonumber
 (L_jf_i, L_jf_i)=&
 (f_i, \overline{L}_j L_j f_i) +
 O\left(\|\xi_{-2}f\| \|L_j f_i\|\right)\\
 \nonumber
 =& (f_i,  L_j\overline{L}_j f_i) +
 O\left(\|\xi_{-2}f\|\big( \|L f\|+\|\overline{L} f\|\big)
 \right)\\
 \label{lparts}
 =& (\overline{L}_jf_i, \overline{L}_j f_i) +
 O\left(\|\xi_{-2}f\|\big( \|L f\|+\|\overline{L} f\|\big)
 \right).
\end{align}
Summing over $i,j$ and using the notation "(s.c.)" for an
arbitrarily small constant, we have
\begin{equation*}
 \|L f \| \lesssim \| \overline{L} f\| + \|\xi_{-2}
f_j\| +(s.c.)\left(\|L f\|+\|\overline{L} f\|\right)
\end{equation*}
which implies
\begin{equation}
\label{relatel}
 \|L f \| \lesssim \| \overline{L} f\| + \|\xi_{-2}
f_j\|.
\end{equation}

 We insert (\ref{bar}) and (\ref{star}) into the
 calculation of the norms $\|\mdbar f\|^2 + \|\mdbar^{\ast} f
 \|^2$ and follow \cite{CS} to write
\begin{equation}
 \label{cseq}
\|\mdbar f\|^2 + \|\mdbar^{\ast} f
 \|^2 = \|\overline{L} f\|^2 -
  \sum_{j,k} \left( \overline{L}_k f_j, \overline{L}_j f_k\right)
  +\sum_{j,k} \left( L_j f_j, L_k f_k\right) +error,
 \end{equation}
where here and below "error" will refer to terms which can be
estimated by
\begin{equation*}
\|\mdbar f\|^2 +
 \|\mdbar^{\ast} f\|^2
  +(s.c)\left(\|\overline{L} f\|^2
 +\|L f\|^2\right)
  + \|\xi_{-2}f\|^2.
\end{equation*}
"(s.c.)" is the notation for an arbitrarily small constant.

The middle terms on the right hand side of (\ref{cseq}) can be
related
 as in (\ref{lparts}):
\begin{equation*}
 (L_jf_j, L_kf_k)=
 (\overline{L}_k f_j, \overline{L}_j f_k)
 +error.
\end{equation*}

Combining we get
\begin{equation*}
\| \mdbar f\|^2+\| \mdbar^{\ast}f\|^2
 =
 \|\overline{L} f\|^2
+error.
\end{equation*}

Thus we obtain the analogue of the Morrey-Kohn-H\"{o}rmander
estimate applied to $C^1$ $(0,1)$ forms with compact support in a
neighborhood of the singular point:
\begin{align*}
  \|   \mdbar f\|^2 + \|
   \mdbar^{\ast} f\|^2
 \gtrsim&
  \| \overline{L} f \|^2
   - (\|\mdbar f\|^2 + \|\mdbar^{\ast} f\|^2 +\|\xi_{-2}f\|^2)
  - (s.c)\left(\|\overline{L}f\|^2+\|Lf\|^2\right)\\
 \gtrsim&
  \| \overline{L} f \|^2
   - (\|\mdbar f\|^2 + \|\mdbar^{\ast} f\|^2 +\|\xi_{-2}f\|^2)
  - (s.c)\|\overline{L}f\|^2\\
  \gtrsim&  \| \overline{L} f \|^2
   - (\|\mdbar f\|^2 + \|\mdbar^{\ast} f\|^2 +\|\xi_{-2}f\|^2),
\end{align*}
 where (\ref{relatel}) is used in the second step.
Hence we have the estimate for $\|\overline{L}f\|$:
\begin{equation}
 \label{weightnorm}
 \| \overline{L} f \|^2
 \lesssim \| \mdbar f\|^2 + \| \mdbar^{\ast} f\|^2
  +\|\xi_{-2}f\|^2.
\end{equation}

 So far the analysis which led to the estimate in
 (\ref{weightnorm}) was done with forms with compact support
 in a neighborhood of the singularity.  To
 globalize such estimates
we can take a partition of unity and consider
 separately those forms with support in a neighborhood of the
 singularity and those whose support intersects the boundary.
 For forms which are supported in a neighborhood intersecting the
 boundary
 we rely on the condition of strict
 pseudoconvexity to handle boundary terms.  In the case of
 a form $f$ supported in a neighborhood, $U$, of a boundary point
 of $\partial X$, the estimates can be read directly from the
 classical estimates of strictly pseudoconvex domains in complex
 manifolds \cite{CS}:
for
 $f\in \mbox{dom}(\mdbar)\cap \mbox{dom}(\mdbar^{\ast})\cap
C^2_{(0,1)}(U)$, we have
\begin{equation*}
 \sum_{i,j=1}^2\int_{U\cap\partial X}
  \rho_{jk}
   f_j \overline{f}_k dS+
 \| \overline{L} f \|^2
 \lesssim \| \mdbar f\|^2 + \| \mdbar^{\ast} f\|^2.
\end{equation*}

Using
\begin{equation*}
\sum_{i,j=1}^2\int_{U\cap\partial X}
  \rho_{jk}
   f_j \overline{f}_k dS
     >\|f\|_{L^2(\partial X)}^2
\end{equation*}
and a density argument we have the estimate for
 forms in $\mbox{dom}(\mdbar)\cap\mbox{dom}(\mdbar^{\ast})$
  supported in neighborhood of boundary point:
\begin{equation}
\label{nearbndry}
 \| \overline{L} f \|^2
 \le \| \mdbar f\|^2 + \| \mdbar^{\ast} f\|^2.
\end{equation}

Using a partition of unity and the
 corresponding contributions from (\ref{weightnorm})
 and (\ref{nearbndry})
 leads to the
\begin{prop}
 \label{c1est}
Let $f\in \mbox{dom}(\mdbar)\cap \mbox{dom}(\mdbar^{\ast})\cap
C^1_{(0,1)}(X)$.  Then we have the estimate
\begin{equation*}
\| \overline{L} f \|^2
 \le \| \mdbar f\|^2 + \| \mdbar^{\ast} f\|^2
  +\|\xi_{-2}f\|^2
 .
\end{equation*}
\end{prop}

Note that
$$\| \xi_{-2} f\|^2 \sim \|f\|^2_{L^{2,-2}}.$$

\section{Approximation by smooth forms}

The goal of this section is to relax the condition of Proposition
\ref{c1est} so as to apply the proposition to all $f\in
L^{2,-2}_{(0,1)}(X)\cap\mbox{dom}(\mdbar)\cap
 \mbox{dom}(\mdbar^{\ast})$.
Using a partition of unity we will assume in this section that
 $f$ has support near the origin (approximation of
  forms supported near the
 strictly pseudoconvex boundary may be handled as in the classical
 theory).
We thus let $U\subset \mathbb{C}^2$ be a neighborhood of the
origin containing the support of $f$ and
 $x=(x_1,\ldots,x_4)$ the real coordinates in $U$.  We take $\chi(x)\in
 C^{\infty}_0(U)$ as an approximation of the identity:
\begin{equation*}
\int_U \chi(x) d{\bf x}^4 =1,
\end{equation*}
where $d{\bf x}^4=dx_1\cdots dx_4$. Furthermore,
\begin{equation*}
\chi_{\varepsilon} (x) = \frac{1}{\varepsilon^4}
 \chi \left(\frac{x}{\varepsilon} \right).
\end{equation*}

Our immediate goal for a given
 $f\in
L^{2,-2}_{(0,1)}(X)\cap\mbox{dom}(\mdbar)\cap
\mbox{dom}(\mdbar^{\ast})$ is to find a sequence of $C^1$ forms to
which Proposition \ref{c1est} can be applied and so that in the
limit the inequality
\begin{equation*}
\| \overline{L} f \|^2
 \le \| \mdbar f\|^2 + \| \mdbar^{\ast} f\|^2
  +\|\xi_{-2}f\|^2
\end{equation*}
is obtained.

To this end, for
$f=f_1\overline{\omega}_1+f_2\overline{\omega}_2$,
 we define
\begin{align*}
 f_{\varepsilon}&=(f_1\ast \chi_\varepsilon)\overline{\omega}_1
  + (f_2\ast \chi_\varepsilon)\overline{\omega}_2,
 \end{align*}
where convolution is taken with respect to the Euclidean
 volume element on $\mathbb{R}^4$, i.e.
\begin{equation*}
f_j=\int f_j(x-y) \chi_{\varepsilon}(y) d{\bf y}^4
 \qquad j=1,2.
\end{equation*}

Let $D$ be a differential operator of the form
$$D=\beta(u,v) \frac{\partial}{\partial x_i}$$
such that $\gamma^2 \beta$ is Lipschitz continuous.  This is the
case e.g. if $\beta$ is $C^1$-smooth outside the origin and
$|\beta|\lesssim 1/\gamma$ as with components of our vector fields
in
 (\ref{vfform}).

\begin{lemma}
 \label{gdf}
Let $f\in L^{2,-2}(X)$ such that $Df\in L^2(X)$. Then
\begin{eqnarray}\label{eq:1}
Df_\epsilon \rightarrow Df\ \ \ \mbox{ in } \ \ L^2(X).
\end{eqnarray}
\end{lemma}

\begin{proof}
Let us first recall that for a function $g$, we have that
$$\|g\|_{L^2(X)} \sim \|\gamma^2 g\|_{L^2(\C^2)},$$
where we denote by $\|\cdot\|_{L^2(\C^2)}$ the standard Euclidean
$L^2$-norm in $\C^2$. More generally, we have
$$\|g\|_{L^{2,k}(X)} = \|\gamma^k g\|_{L^2(X)} \sim \|\gamma^{2+k} g\|_{L^2(\C^2)}$$
for any weight $k\in\R$.

Hence, we note that
$$f,\  \gamma^2 Df \in L^2(\C^2).$$
Thus $f_\epsilon$ is in fact well defined (and smooth).

We also note that \eqref{eq:1} is equivalent to
\begin{eqnarray*}\label{eq:2}
\gamma^2 Df_\epsilon \rightarrow \gamma^2 Df\ \ \ \mbox{ in } \ \
L^2(\C^2).
\end{eqnarray*}

So, it makes sense to study the operator $\gamma^2 D$ which we may
write as
$$P:= a(u,v) \frac{\partial}{\partial x_i},$$
where $a$ is Lipschitz continuous by assumption. Our problem is reduced to
showing that
\begin{eqnarray*}\label{eq:3}
P f_\epsilon \rightarrow Pf \ \ \ \mbox{ in } \ \ L^2(\C^2)
\end{eqnarray*}
for $f\in L^2(\C^2)$ such that $Pf\in L^2(\C^2)$.

But that holds by the well-known Friedrichs extension lemma (see
e.g. \cite{CS}, Corollary D.2, which holds for Lipschitz
continuous coefficients as is immediate from the proof in \cite{CS}).
\end{proof}

We now apply Lemma \ref{gdf} to $\mdbar f$.
  Here we work with the
differential operator $D$ on a form $g=g_1 \omega_1 + g_2\omega_2$
defined by
\begin{equation*}
Dg=\overline{L}_1 g_2 -
 \overline{L}_2 g_1.
\end{equation*}
 Then Lemma \ref{gdf} shows that, for
  $f\in L^{2,-2}_{(0,1)}(X)$, $D f_{\varepsilon} \rightarrow
 Df$.
 It is trivial that $\xi_{-2}f_{\varepsilon}
 \overset{L^2}{\rightarrow} \xi_{-2}f$ for
  $f\in L^{2,-2}_{(0,1)}(X)$, and so we have $\mdbar f_{\varepsilon}
 \overset{L^2}{\rightarrow} \mdbar f$.
Furthermore, the proof of Lemma \ref{gdf} can be applied to the
first order differential operator associated with the operator
$\mdbar^{\ast}$ and so as a corollary, we have
\begin{cor} Let $f\in L^{2,-2}_{(0,1)}(X)
\cap
 \mbox{dom}(\mdbar)\cap \mbox{dom}(\mdbar^{\ast})$.  Then
\label{gdbar}
\begin{align*}
&\mdbar f_{\varepsilon}
 \overset{L^2}{\rightarrow} \mdbar f\\
&\mdbar^{\ast} f_{\varepsilon}
 \overset{L^2}{\rightarrow} \mdbar^{\ast} f.
\end{align*}
\end{cor}

Given
 $f\in L^{2,-2}_{(0,1)}(X)$ we can now apply Proposition \ref{c1est}
  to $f_{\varepsilon}$ to obtain
\begin{equation*}
\|  \overline{L} f_{\varepsilon} \|^2
 \le \| \mdbar f_{\varepsilon}\|^2 +
\|  \mdbar^{\ast} f_{\varepsilon}\|^2
  +\|\xi_{-2}f_{\varepsilon}\|^2
  .
\end{equation*}

Letting ${\varepsilon}\rightarrow 0$ and using Corollary \ref{gdbar}
 shows that for $f\in \mbox{dom}(\mdbar)\cap \mbox{dom}(\mdbar^{\ast})\cap
L^{2,-2}_{(0,1)}(X)$, we have the estimate
\begin{equation*}
\| \overline{L} f \|^2
 \le \|  \mdbar f\|^2 + \| \mdbar^{\ast} f\|^2
  +\|\xi_{-2}f\|^2
\end{equation*}
 from which we
finally conclude the generalization of Proposition \ref{c1est}:
\begin{thrm}
\label{weightest}
 Let $f\in
L^{2,-2}_{(0,1)}(X)\cap\mbox{dom}(\mdbar)\cap
\mbox{dom}(\mdbar^{\ast})$.  Then we have the estimate
\begin{equation*}
\| \overline{L} f \|^2
 \le \| \mdbar f\|^2 + \| \mdbar^{\ast} f\|^2
  +\|\xi_{-2}f\|^2.
\end{equation*}
\end{thrm}
We again note that $\|\xi_{-2}f\|^2 \sim \|f\|^2_{L^{2,-2}}$.

\section{Intermediate Sobolev norms}
\label{subnorms}

For the time being we work on a neighborhood of the singularity.
Then Theorem \ref{weightest} can be combined with (\ref{relatel}),
\begin{equation*}
\|Lf\| \lesssim \| \overline{L} f\| + \|\xi_{-2} f\|,
\end{equation*}
to show
\begin{equation}
\label{LL}
 \|  Lf \|^2 + \|\overline{L}f\|^2
 \lesssim \| \mdbar f\|^2 + \| \mdbar^{\ast} f\|^2
  +\|\xi_{-2}f\|^2
\end{equation}
for $f\in L^{2,-2}_{(0,1)}(X)\cap\mbox{dom}(\mdbar)\cap
\mbox{dom}(\mdbar^{\ast})$ with support in a neighborhood of the
singularity.

We now define Sobolev spaces on our space $X$.  We recall the
 association of $X$
with the Hermitian space $(B,g)$ from Section \ref{prelim}, where
 $B=\{ |v|^4+|w|^4<1\}$ and $g\sim \gamma^4$ is the metric
 (\ref{metric}).
We let
 $v=x_1+ix_2$ and $w=x_3+ix_4$.
 For $k$ an integer, $W^k(X)$ is the
space of functions whose derivatives with respect to the
coordinates $x_j$ of order less than or equal to $k$ are in
$L^2(X)$. Thus, for $f\in W^k(X)$
\begin{equation}
\label{sobolev1}
\|f\|_{W^k(X)} \sim
 \sum_{|l|\le k} \int_B \left|
  \frac{\partial^{|l|}}{\partial x^l} f\right|^2
  \gamma^4 d{\mathbf x}^4
 ,
\end{equation}
where $l=(l_1,\ldots,l_4)$ is a multi-index of length 4 and
\begin{equation*}
\frac{\partial^{|l|}}{\partial x^l}
 = \frac{\partial^{l_1}}{\partial x_1^{l_1}}\cdots
  \frac{\partial^{l_4}}{\partial x_4^{l_4}}.
\end{equation*}
For a $(0,1)$-form, $f=f_1\overline{\omega}_1
 +f_2\overline{\omega}_2$, the $W^k_{(0,1)}$-norm of $f$ is
 equivalent to the sum of the $W^k$-norms of $f_1$ and $f_2$.

 We note that
 any derivative in the $x$ coordinates above
  is a combination (by multiplication by bounded
  functions) of the vector fields $\gamma L_j$
 and $\gamma \overline{L}_j$, for $j=1,2$.  Given a
  form  $f\in L^{2,-1}_{(0,1)}(X)$, then
   $\gamma f\in L^{2,-2}_{(0,1)}(X)$, and we can thus apply
  (\ref{LL}) to $\gamma f$ to bound the
  $W^1(X)$ norm of a $(0,1)$ form according to
\begin{thrm}
\label{w1}
  Let $f\in
L^{2,-1}_{(0,1)}(X)\cap\mbox{dom}(\mdbar)\cap
\mbox{dom}(\mdbar^{\ast})$ with support in a neighborhood of the
singularity. Then we have the estimate
\begin{equation*}
\|  f \|_{W^1(X)}^2
 \lesssim \|\gamma \mdbar f\|^2 + \|\gamma \mdbar^{\ast} f\|^2
  +\|\xi_{-1}f\|^2
  .
\end{equation*}
\end{thrm}
Our subelliptic estimates can be viewed as a
 trade-off of a decrease in derivatives for less stringent
 hypotheses on to which $L^p$ spaces a given form may belong.
 Our results show that the restriction of
  a form $f$ from more $L^p$ spaces for $p>2$,
  allows for higher orders of derivatives of $f$
 to be estimated.

For non-integer Sobolev spaces we use the Fourier multiplier
 operator on $\mathbb{R}^4$ with the symbol
  $(1+|\zeta|^2)^{\epsilon/2}$.  $\Lambda^{\epsilon} f$ is then
  defined through its Fourier transform:
\begin{equation*}
\widehat{\Lambda^{\epsilon} f}(\zeta)
 =(1+|\zeta|^2)^{\epsilon/2} \widehat{f}(\zeta).
\end{equation*}
 $W^{\epsilon}(X)$ is the restriction to the
  domain $B$ of the
  space of functions on $\mathbb{R}^4$ which satisfy
\begin{equation}\label{sobolev2}
 \int_{\mathbb{R}^4}
  \left|\Lambda^{\epsilon}f\right|^2\gamma^4 d{\mathbf x}^4
   <\infty.
\end{equation}
The rest of this section is dedicated to proving an estimate for a
norm for $W^{\epsilon}(X)$.  First we note that for a function
supported outside a neighborhood of the singularity at the origin,
its weighted $L^2$ norm is equivalent to an unweighted norm.  If
we define
\begin{equation*}
\triangle_{\lambda}:= \{|v|^4+|w|^4<\lambda\}
\end{equation*}
for $\lambda<1$, then
 $\|f\|_{W^{\epsilon}(X\setminus\triangle_{\lambda})}$ is
 equivalent to the usual $W^{\epsilon}$ norm on the domain
 $B\setminus\triangle_{\lambda}\subset\mathbb{C}^2$.  With
  $\lambda_1<\lambda_2<1$ we let $\varphi_{\lambda_1,\lambda_2}\in
  C^{\infty}(\mathbb{R}^4)$ with support in
  $\triangle_{\lambda_2}$ and such that
  $\varphi_{\lambda_1,\lambda_2}\equiv
  1$ in $\triangle_{\lambda_1}$.
  Then
\begin{equation*}
\|f\|_{W^{\epsilon}(X)}
 \lesssim \|f\|_{W^{\epsilon}(X\setminus\triangle_{\lambda_1})}
  + \|\varphi_{\lambda_1,\lambda_2}f\|_{W^{\epsilon}(X)},
\end{equation*}
where, from above,
\begin{equation*}
\|\varphi_{\lambda_1,\lambda_2}f\|_{W^{\epsilon}(X)}
 \sim
\int_{\mathbb{R}^4}
  \left|\Lambda^{\epsilon}
  \varphi_{\lambda_1,\lambda_2}f\right|^2
  \gamma^4 d{\mathbf x}^4.
\end{equation*}

Without further notation we assume until otherwise stated that
 the functions we work with have support in a neighborhood of the
 singularity.

 In analogy with our weighted $L^2$ spaces we can define weighted
 versions of the Sobolev spaces in (\ref{sobolev1}):
 \begin{equation*}
 W^{s,k}_{(p,q)}(X)=\left\{ f: \|f\|_{W^{s,k}}^2=
  \sum_{|l|\le s} \int_B \gamma^{2k} \left|
  \frac{\partial^{|l|}}{\partial x^l} f\right|^2
  \gamma^4 d{\mathbf x}^4 <\infty \right\},
\end{equation*}

 Given $0\le \epsilon \le 1$, if $f\in W^{1,1-\epsilon}(X)$, we can write
\begin{align}
\nonumber
 \| f\|_{W^{\epsilon}(X\setminus \triangle_{\lambda})}^2 &\leq
 C_{\lambda}
  \|f\|_{W^1(X\setminus \triangle_{\lambda})}^2\\
  &\le C_{\lambda}' \|f\|_{W^{1,1-\epsilon}(X)}^2
   \label{lemest}
\end{align}
with constants $C_\lambda, C_{\lambda}'>0$ which do not depend on
$f$.

 The idea to obtain a $W^{\epsilon}$ estimate for functions
  $f\in W^{1,1-\epsilon}(X)$ is to let $\lambda\rightarrow 0$
   in (\ref{lemest}).
However, we note that the constant of the inequality in
(\ref{lemest}) may depend on $\lambda$. The rest of the section
shows that
 when using (\ref{lemest}) to estimate
  $\|f\|_{W^{\epsilon}(X)}$, the
$\lambda$-dependence of the constant may be controlled by
 adjusting an $L^2$ norm on the right hand side.

For $\alpha<1/2$ we can use a simple construction to find
$\varphi_{\alpha}=
 \varphi_{\alpha,2\alpha}$ so that
$\varphi_{\alpha}\equiv 1$ in $\triangle_{\alpha}$ and
 $\varphi_{\alpha}\equiv 0$ in $X\setminus \triangle_{2\alpha}$
 with the property that
\begin{equation*}
| \widehat{\Lambda^4\varphi}_{\alpha}(\xi) | \le c,
\end{equation*}
with $c$ a constant independent of $\alpha$.

From above we have
\begin{equation}
 \| f\|_{W^{\epsilon}(X)}
 \le
  \| f\|_{W^{\epsilon}(X\setminus\triangle_{\alpha})}
   + \|\varphi_{\alpha} f\|_{W^{\epsilon}(X)}
   \label{combine}
 .
\end{equation}
For the second term on the right hand side we can consider
 $f$ and $\varphi_{\alpha}$ to be defined on all of
  $\mathbb{C}^2$ by extension by 0 outside of $B$.
 We then use
\begin{align}
\nonumber
 \|\varphi_{\alpha} f\|_{W^{\epsilon}(X)}
  & \le
   \| f \Lambda^{\epsilon}\varphi_{\alpha} \|_{L^2(X)}
    + \| \varphi_{\alpha}\Lambda^{\epsilon} f\|_{L^2(X)}\\
    \label{2ndterm}
 &\le \| f \Lambda^{\epsilon}\varphi_{\alpha} \|_{L^2(X)}
  +c(\alpha) \|f\|_{W^{\epsilon}(X)},
\end{align}
where $c(\alpha)\rightarrow 0$ as $\alpha\rightarrow 0$.

We concentrate on the term $ \Lambda^{\epsilon}\varphi_{\alpha}
 = \Lambda^{\epsilon-4}\circ\Lambda^4 \varphi_{\alpha}
$.
\begin{align*}
|\widehat{\Lambda^{\epsilon}\varphi_{\alpha}}|
 &= (1+|\xi|^2)^{(\epsilon-4)/2}
  |\widehat{\Lambda^4 \varphi_{\alpha}}|\\
  &\le c (1+|\xi|^2)^{(\epsilon-4)/2},
\end{align*}
from which we have
 $\widehat{\Lambda^{\epsilon}\varphi_{\alpha}}\in L^p(\mathbb{R}^4)$
  for $p> 4/(4-\epsilon)$ and
 $\Lambda^{\epsilon}\varphi_{\alpha}\in L^q(\mathbb{R}^4)$ for
  $ q<4/\epsilon$.  Hence we have
\begin{equation*}
\| f \Lambda^{\epsilon}\varphi_{\alpha} \|_{L^2(X)}
 \lesssim
  \left(\int_{X} |f|^{2s'}dV_X \right)^{\frac{1}{2s'}}
\left(\int_{\mathbb{R}^4}
|\Lambda^{\epsilon}\varphi_{\alpha}|^{2s} \gamma^4 d{\bf x}^4
\right)^{\frac{1}{2s}},
\end{equation*}
where $s=q/2$ and
 $s'=q/(q-2)$, i.e.
\begin{equation*}
\| f \Lambda^{\epsilon}\varphi_{\alpha} \|_{L^2(X)}
 \lesssim \| f\|_{L^p(X)},  \qquad p>\frac{4}{2-\epsilon}.
\end{equation*}

Returning to (\ref{2ndterm}), we have for some $\alpha$ small enough
(which we fix from now on):
\begin{align*}
\|\varphi_{\alpha} f\|_{W^{\epsilon}(X)}
  &\lesssim \| f \Lambda^{\epsilon}\varphi_{\alpha} \|_{L^2(X)}
  +(s.c.) \|f\|_{W^{\epsilon}(X)}\\
 &\lesssim \| f\|_{L^p(X)} + (s.c.)\|f\|_{W^{\epsilon}(X)}.
\end{align*}

If we assume $f\in L^{p}_{(0,1)}(X)\cap
 \mbox{dom}(\mdbar)\cap \mbox{dom}(\mdbar^{\ast})$, then
  $\| \xi_{-\epsilon} f\|_{L^2(X)}
 \lesssim
  \|f\|_{L^p(X)}$
for $p>4/(2-\epsilon)$ and
 $f\in L^{2,-\epsilon}_{(0,1)}(X)\cap
 \mbox{dom}(\mdbar)\cap \mbox{dom}(\mdbar^{\ast})$.
   Furthermore, $\gamma^{1-\epsilon}
  f\in L^{2,-1}_{(0,1)}(X)$, and
we can apply Theorem \ref{w1} to show
 $\gamma^{1-\epsilon} f \in W^1(X)$.  Since
  $f\in L^{2,-\epsilon}(X)$, we also have
   $f\in W^{1,1-\epsilon}(X)$ and we can apply
  (\ref{lemest}) to estimate the term
$\| f\|_{W^{\epsilon}(X\setminus\triangle_{\alpha})}$
 in (\ref{combine}).  Finally, we can write
\begin{align*}
\| f\|_{W^{\epsilon}(X)} &\lesssim
 \| f\|_{W^{\epsilon}(X\setminus\triangle_{\alpha})}
 + \| f\|_{L^p(X)}
  + (s.c.)\|f\|_{W^{\epsilon}(X)}\\
 &\lesssim
  \|f\|_{W^{1,1-\epsilon}(X)}
  + \| f\|_{L^p(X)}
  + (s.c.)\|f\|_{W^{\epsilon}(X)}.
\end{align*}
  We conclude the
following
\begin{prop}
\label{subprop}
 Let $0\le \epsilon\le 1$ and $p>4/(2-\epsilon)$.  For
 $f\in L^{p}_{(0,1)}(X)\cap
 \mbox{dom}(\mdbar)\cap \mbox{dom}(\mdbar^{\ast})$
 with support in a neighborhood of the singularity
 then $f\in W^{\epsilon}(X)$.
   The norm $\|f\|_{W^{\epsilon}(X)}$ is bounded by
\begin{equation*}
\|f\|_{W^{\epsilon}(X)}\lesssim
  \|f\|_{W^{1,1-\epsilon}(X)}
   +\|f\|_{L^p(X)}.
\end{equation*}
\end{prop}

\section{Subelliptic estimates}
 We apply Proposition \ref{subprop}
  to each component, $f_j$, $j=1,2$,
  of a given form
  $f\in L^{p}_{(0,1)}(X)\cap
 \mbox{dom}(\mdbar)\cap \mbox{dom}(\mdbar^{\ast})$
   with support in a neighborhood of the singularity.
 To bound the terms
\begin{equation*}
\int_X \gamma^{2-2\epsilon} | \nabla f_j|^2 dV
\end{equation*}
we use
\begin{align*}
\int_X \gamma^{2-2\epsilon} | \nabla f|^2 dV
 &\lesssim \|\gamma^{2-\epsilon} \overline{L}f\|^2 +
 \|\gamma^{2-\epsilon} Lf\|^2\\
  &\lesssim
    \| \gamma \overline{L}\left( \gamma^{1-\epsilon}f\right) \|^2 +
 \| \gamma L \left( \gamma^{1-\epsilon}f\right)\|^2 + \|
 \xi_{-\epsilon} f\|^2.
\end{align*}
Using the estimate $\| \xi_{-\epsilon} f\|_{L^2(X)}
 \lesssim
  \|f\|_{L^p(X)}$
for $p>4/(2-\epsilon)$ and Proposition \ref{subprop} we have
\begin{equation*}
\| f\|_{W^{\epsilon}(X)}^2
 \lesssim
 \| \gamma \overline{L}\left( \gamma^{1-\epsilon}f\right) \|^2 +
 \| \gamma L \left( \gamma^{1-\epsilon}f\right)\|^2 + \| f\|_{L^p(X)}^2
\end{equation*}

As in Section \ref{subnorms} above, given $0\le\epsilon\le1$ and
  a form $f\in L^{p}_{(0,1)}(X)\cap
 \mbox{dom}(\mdbar)\cap \mbox{dom}(\mdbar^{\ast})$ for $p>4/(2-\epsilon)$,
  we have
  $\gamma^{1-\epsilon}
  f\in L^{2,-1}_{(0,1)}(X)$, and
we can apply Theorem \ref{w1} to
 the form $\gamma^{1-\epsilon}f$ in place of $f$ to conclude
\begin{align}
 \nonumber
\|f\|_{W^{\epsilon}(X)}^2
 &\lesssim
\|\gamma \mdbar \left(\gamma^{1-\epsilon} f\right)\|^2 + \| \gamma
\mdbar^{\ast} \left(\gamma^{1-\epsilon} f\right)\|^2
  +\|f\|_{L^p(X)}^2\\
  \label{subcom}
  &\lesssim
\| \gamma^{2-\epsilon} \mdbar  f\|^2 + \|\gamma^{2-\epsilon}
\mdbar^{\ast} f\|^2
  +\|f\|_{L^p(X)}^2.
\end{align}
 We can now use a partition of unity to remove the assumption of
 support near the singularity.  We use
   estimates given by (\ref{subcom}) for
 forms with support in a neighborhood of the singularity, and
 subelliptic $1/2$-estimates for forms with support near the
 boundary $\partial X$ (Theorem 5.1.2 \cite{CS}):
 \begin{equation}
 \label{subbnd}
 \|f\|_{W^{\epsilon}(X)} \lesssim \|f\|_{W^{1/2}(X)}
  \lesssim \|\mdbar f\|+\|\mdbar^{\ast}f\|,
   \qquad 0\le\epsilon\le1/2.
 \end{equation}
 Combining (\ref{subcom}) and (\ref{subbnd}),
 we arrive at
 our
subelliptic estimate which we write in the form of
\begin{thrm}
 Let $0\le \epsilon\le 1/2$, $p>4/(2-\epsilon)$,
  and $f\in L^{p}_{(0,1)}(X)\cap
 \mbox{dom}(\mdbar)\cap \mbox{dom}(\mdbar^{\ast})$.  Then the
 following
 subelliptic estimate holds
\begin{equation*}
 \| f\|_{W^{\epsilon}(X)}
  \lesssim
\|  \mdbar f\|+ \| \mdbar^{\ast} f\|
 +\| f\|_{L^{p}(X)}
  .
\end{equation*}
\end{thrm}

\end{document}